\documentclass[preprint,12pt]{elsarticle}

\usepackage{amssymb}
\usepackage{amsmath,mathrsfs,amsfonts}
\usepackage{mathtools}
\usepackage{amsthm}
\usepackage{xcolor}

\usepackage{cases}
\usepackage{multirow}
\usepackage{comment}
\usepackage{hyperref}
\usepackage{soul}
\usepackage{enumitem}

\newtheorem{theorem}{Theorem}
\newtheorem{corollary}{Corollary}
\newtheorem{remark}{Remark}
\newtheorem{lemma}{Lemma}

\newcommand{\proj}{\operatorname{proj}}

\newcommand{\bqo}{\operatorname{BQO}}

\newcommand{\milo}{\operatorname{MILO}}
\newcommand{\vmilo}{\operatorname{V-MILO}}
\newcommand{\emilo}{\operatorname{E-MILO}}
\newcommand{\remilo}{\operatorname{RE-MILO}}

\newcommand{\misdo}{\operatorname{MISDO-I}}
\newcommand{\misdos}{\operatorname{MISDO-II}}

\newcommand{\suchthat}{\;\ifnum\currentgrouptype=16 \middle\fi|\;}

\makeatletter
\newcommand{\vast}{\bBigg@{4}}
\newcommand{\Vast}{\bBigg@{2.5}}
\makeatother

\newcommand{\Biggpl}{\mathopen{\raisebox{-1ex}{$\Vast($}}}
\newcommand{\Biggpr}{\mathopen{\raisebox{-1ex}{$\Vast)$}}}

\newcommand{\Biggfl}{\mathopen{\raisebox{-.5ex}{$\Bigg($}}}
\newcommand{\Biggfr}{\mathopen{\raisebox{-.5ex}{$\Bigg)$}}}

\newcommand{\Bigfl}{\mathopen{\raisebox{-.5ex}{$\Big($}}}
\newcommand{\Bigfr}{\mathopen{\raisebox{-.5ex}{$\Big)$}}}

\newcommand{\relaxbqov}{\operatorname{\mathcal{R}^y_{BQO}}}

\newcommand{\relaxvmilo}{\operatorname{\mathcal{R}_{V-MILO}}}

\newcommand{\relaxbqoe}{\operatorname{\mathcal{R}^z_{BQO}}}
\newcommand{\relaxemilo}{\operatorname{\mathcal{R}_{E-MILO}}}
\newcommand{\relaxmisdo}{\operatorname{\mathcal{R}_{MISDO-I}}}

\newcommand{\relaxmisdoii}{\operatorname{\mathcal{R}_{MISDO-II}}}

\newcommand{\relaxbqom}{\operatorname{\mathcal{R}^Z_{BQO}}}

\newcommand{\relaxbqomii}{\operatorname{\mathcal{R}^{\bar{Z}}_{BQO}}}

\usepackage[labelformat=simple]{subcaption}
\usepackage[font=scriptsize]{caption}
\graphicspath{{images/}}

\journal{Operations Research Letters}

\begin{document}

\begin{frontmatter}



\title{On relaxations of the max $k$-cut problem formulations}


\author[1]{Ramin Fakhimi}

\author[2]{Hamidreza Validi}

\author[3]{Illya V. Hicks}

\author[1]{Tam\'as  Terlaky}

\author[1]{Luis F. Zuluaga}

\address[1]{organization={Industrial and Systems Engineering, Lehigh University},
state={PA},
country={U.S.}}
            
\address[2]{organization={Industrial, Manufacturing \& Systems Engineering, Texas Tech University},
state={TX},
country={U.S.}}  
            
\address[3]{organization={Computational Applied Mathematics \& Operations Research, Rice University},
state={TX},
country={U.S.}}             

\begin{abstract}
A tight continuous relaxation is a crucial factor in solving mixed integer formulations of many NP-hard combinatorial optimization problems.
The (weighted) max $k$-cut problem is a fundamental combinatorial optimization problem with multiple notorious mixed integer optimization formulations.  
In this paper, we explore four existing mixed integer optimization formulations of the max $k$-cut problem. 
Specifically, we show that the continuous relaxation of a binary quadratic optimization formulation of the problem is: (i) stronger than the continuous relaxation of two mixed integer linear optimization formulations and (ii) at least as strong as the continuous relaxation of a mixed integer semidefinite optimization formulation. 
%
%
%
%
We also conduct a set of experiments on multiple sets of instances of the max $k$-cut problem using state-of-the-art solvers that empirically confirm the theoretical results in item (i). 
Furthermore, these numerical results illustrate the advances in the efficiency of global non-convex quadratic optimization solvers and more general mixed integer nonlinear optimization solvers.
As a result, these solvers provide a promising option to solve combinatorial optimization problems. 
%
%
Our codes and data are available on GitHub.
\end{abstract}


\end{frontmatter}

\section{Introduction}
\label{intro}
The continuous relaxation of a mixed integer optimization formulation plays a fundamental role in the efficient solution process of not only linear formulations, but also non-linear ones of a mixed integer optimization problem~\citep{muller2022}.
The (weighted) max $k$-cut problem is a fundamental NP-hard combinatorial optimization problem~\citep{frieze1997, Papadimitriou1991_Optimization} with multiple mixed integer linear optimization formulations that suffer either from weak relaxation or large size.
The max $k$-cut problem has a wide range of applications, including but not limited to statistical physics~\citep{Barahona1988_Application, de1995_exact}, gas and power networks~\citep{hojny2020_mixed}, data clustering~\citep{poland2006}, and scheduling~\citep{carlson1966_scheduling}.
Given a graph $G=(V,E)$ with edge weights $w$ and a positive integer number $k \ge 2$, the max $k$-cut problem seeks to find at most $k$ partitions such that the weights of edges with endpoints in different partitions are maximized.


%

Motivated by the considerable effect of the continuous relaxation strength in solving mixed integer optimization formulations of the max $k$-cut to optimality, we discuss multiple known optimization formulations of the max $k$-cut problem in the literature: (i) a binary quadratic optimization ($\bqo$) formulation~\citep{carlson1966_scheduling}; (ii) a vertex-based mixed integer linear optimization ($\vmilo$) formulation; (iii) an edge-based  mixed integer linear optimization ($\emilo$) formulation~\citep{Chopra1993_partition}; and (iv) two mixed integer semidefinite optimization (MISDO) formulations~\citep{frieze1997_improved,Eisenblater2002_Semidefinite}.
We prove that the continuous relaxation of the $\bqo$ formulation is: (i) stronger than that of the $\vmilo$ and the $\emilo$ formulations and (ii) at least as strong as the mixed integer semidefinite optimization formulations.  
%
Further, we conduct a set of computational experiments to assess our theoretical results in practice.
Thanks to the recent advancements of state-of-the-art solvers (e.g., solvers for non-convex quadratic optimization formulations to optimality by Gurobi 10.0.0), we can solve the BQO formulation with these solvers. 
The numerical results support our theoretical ones regarding the strength of the continuous relaxation of the mixed integer programming formulations of the max $k$-cut problem.
The continuous relaxation of the BQO formulation provides a tighter upper bound compared to the other MILO formulations.
It also provides a high-quality upper bound for large-scale instances of the problem while the continuous relaxation of the MISDO formulations struggle to achieve.

\section{Mixed Integer Optimization Formulations}
Motivated by solving a scheduling problem,~\citet{carlson1966_scheduling} proposed a $\bqo$ formulation for the max $k$-cut problem. 
Let $n \coloneqq |V|$ and $m \coloneqq |E|$ respectively be the number of vertices and edges of graph $G = (V, E)$. 
Furthermore, we define $P \coloneqq \{1, \dots, k\}$ as the set of partitions and $w_{uv} \in \mathbb{R}$ as the edge weights for $\{u,v\} \in E$.
For every vertex $v \in V$ and every partition $j \in P$, binary variable $x_{vj}$ is one if vertex $v$ is assigned to partition $j$ and zero otherwise.
Then, the $\bqo$ formulation is as follows.
\begin{subequations}\label{eq: maxKcut_qp}
	\begin{alignat}{3}
	\max \quad &\sum_{\{u,v\} \in E} w_{uv}\Biggfl 1 - \sum_{j \in P} x_{uj}x_{vj} \Biggfr \label{eq: maxKcut_qp_1}\\
	   (\bqo)  \quad \text{s.t.} \quad &\sum_{j \in P} x_{vj} =1 &  \forall v \in V\label{eq: maxKcut_qp_2} \\
	&x \in \{0,1\}^{n \times k} \label{eq: maxKcut_qp_3}.
	\end{alignat}
\end{subequations}
Objective function~\eqref{eq: maxKcut_qp_1} maximizes the number of cut edges, and constraints~\eqref{eq: maxKcut_qp_2} imply that each vertex must be assigned to exactly one partition.
~\citet{carlson1966_scheduling} proved that an optimal solution of the continuous relaxation of $\bqo$ formulation~\eqref{eq: maxKcut_qp} can be converted into an optimal solution of its binary variant with the same objective value.

\begin{theorem}[\citet{carlson1966_scheduling}]\label{theorem: BQO vs its relxation}
An optimal solution of the $\bqo$ formulation~\eqref{eq: maxKcut_qp} is also optimal for its continuous relaxation.
\end{theorem} 
A consequence of Theorem~\ref{theorem: BQO vs its relxation} is that recently improved global non-convex quadratic solvers, as well as more general mixed integer nonlinear optimization solvers, can be used to solve the max $k$-cut problem to optimality. As our numerical results illustrate in Section~\ref{section: computational experiments}, indeed it turns out that this solution approach is promising.


One can linearize the $\bqo$ formulation~\eqref{eq: maxKcut_qp} to develop a $\milo$ formulation of the max $k$-cut problem that is called the vertex-based MILO ($\vmilo$) formulation in this paper. 
For every edge $\{u,v\} \in E$, binary variable $y_{uv}$ is one if the endpoints of edge $\{u,v\}$ belong to different partitions; that is $\{u,v\}$ is a {\em cut edge}, and zero otherwise. 

\begin{subequations}\label{eq:maxKcut}
	\begin{alignat}{3}
	\max \quad &\sum_{\{u,v\} \in E} w_{uv} y_{uv} \label{eq:obj_maxKcut}\\
	\text{s.t.} \quad &\sum_{j\in P} x_{vj} =1 &  \forall v \in V  \label{eq:clusterCons} \\
	(\vmilo)\quad & x_{uj} - x_{vj} \le y_{uv} &   \nonumber\\
	 & x_{vj} - x_{uj} \le y_{uv} &  \forall \{u,v\} \in E, \ j \in P  \label{eq:reduncantCon1} \\
	\quad & x_{uj} + x_{vj} + y_{uv} \le 2 & \quad \forall \{u,v\} \in E, \ j \in P  \label{eq:necessaryCons}\\
	&x \in \{0,1\}^{n \times k}\\
        &y \in \{0,1\}^{m}.
	\end{alignat}
\end{subequations}
Objective function~\eqref{eq:obj_maxKcut} maximizes the total weight of cut edges. 
Constraints~\eqref{eq:clusterCons} imply that every vertex is assigned to exactly one partition.
Constraints~\eqref{eq:reduncantCon1} imply that if endpoints of an edge belong to different partitions, then it is a cut edge.
Constraints~\eqref{eq:necessaryCons} imply that if the endpoints of an edge belong to the same partition, then the edge cannot be a cut edge. 
Despite the reasonable size of formulation~\eqref{eq:maxKcut} having $kn+m$  variables and $n+3km$ constraints, it suffers from weak continuous relaxation and symmetry issues~\citep{hojny2020_mixed}.
%

%
%
Another classical $\milo$ formulation is a large edge-based MILO ($\emilo$) formulation with $\binom{n}{2}$ variables and $3{\binom{n}{3}} +\binom{n}{k+1}$ constraints~\citep{Chopra1993_partition, chopra1995_facets}. 
Although the continuous relaxation of this formulation provides a relatively tight upper bound in practice, classical solvers struggle to solve even medium-size instances of the max $k$-cut problem to optimality~\cite{Wang2020-k-partition}. 
For every set $S$, we employ $\binom{S}{2}$ to denote all subsets of $S$ with size 2.
For every pair of vertices $\{u, v\} \in \binom{V}{2}$, we define binary variable $z_{uv}$ as follows: 
$z_{uv}$ is one if vertices $u$ and $v$ belong to the same partition, and zero otherwise.
\begin{subequations}\label{eq:E-MILO}
	\begin{alignat}{3}
	\max \quad  &\sum_{\{u,v\} \in E} w_{uv}(1 - z_{uv}) \label{eq:obj_E-MILO}\\
 	\text{s.t.} \quad & z_{uv} + z_{vw} \le 1 + z_{uw}  \nonumber\\
 	\quad & z_{uw} + z_{uv} \le 1 + z_{vw}  \nonumber\\
	(\emilo)\quad & z_{vw} + z_{uw} \le 1 + z_{uv} &  \forall \{u,v,w\} \subseteq V  \label{eq:triangle}\\
	\quad & \sum_{\{u,v\} \in \binom{Q}{2}} z_{uv} \ge 1 &  \forall Q \subseteq V, |Q| = k+1\label{eq:atmostk}
	\\
	&z \in \{0,1\}^{\binom{n}{2}}.
	\end{alignat}
\end{subequations}
Objective function~\eqref{eq:obj_E-MILO} maximizes the total weight of cut edges. Constraints~\eqref{eq:triangle} imply that for every set $\{u,v,w\} \subseteq V$, if pairs $\{u,v\}$ and $\{v,w\}$ belong to a partition, then vertices~$u$ and~$w$ also belong to the same partition.
Constraints~\eqref{eq:atmostk} imply that vertex set~$V$ must be partitioned into at most~$k$ partitions.
Because of the large number of constraints~\eqref{eq:atmostk}, one can add them on the fly~\cite{gurobi}. 
\citet{chopra1995_facets} conducted a polyhedral study on the max $k$-cut problem and proposed several facet-defining inequalities for the $\emilo$ formulation.
They also studied the $\vmilo$ and $\emilo$ formulations for the min $k$-cut problem and developed multiple facet-defining inequalities for both formulations~\citep{Chopra1993_partition}. 

Further,~\citet{Wang2020-k-partition} propose a reduced $\emilo$ ($\remilo$) formulation that is constructed as follows: (i) graph $G$ is extended to a chordal graph, (ii) all maximal cliques of the chordal graph are found, (iii) binary variables $z$ are created \emph{only} for the edge set of the chordal graph, and (iv) constraints~\eqref{eq:triangle}--\eqref{eq:atmostk} are added \emph{only} for the maximal cliques. 
The number of variables and constraints in their formulation is fewer than or equal to that of the $\emilo$ formulation~\eqref{eq:E-MILO}.
However, for dense graphs in which the chordalized graph is complete, they are the same as the $\emilo$ formulation.
They show that their formulation outperforms the $\emilo$ formulation when the chordalized graph is sparse.

We also provide an existing mixed integer semidefinite optimization (MIS\-DO) formulation~\citep{frieze1997}. 
For every $(u,v) \in V \times V$, binary variable $Z_{uv}$ is one if vertices $u$ and $v$ belong to the same partition. 
Then the formulation is as follows.
\begin{subequations}\label{eq:MI-SDO}
	\begin{alignat}{3}
	\max\quad &\sum_{\{u,v\} \in E} w_{uv}(1 - Z_{uv}) \label{eq:misdo-obj} \\
 	(\misdo) \quad \text{s.t.} \quad & Z_{vv} = 1 & \forall v \in V \label{eq:misdo-diag}\\
 	\quad & k Z \succeq   e e^T  \label{eq:misdo-psd}\\
	 & Z \in \{0,1\}^{n \times n}.
	\end{alignat}
\end{subequations}

Motivated by the superiority of the computational performance of an alternative MISDO formulation that employs $\{-1/(k-1), 1\}$ variables instead of binary ones (e.g., see~\citet{Eisenblater2002_Semidefinite} and~\citet{de2019}), we provide the following remark. 
\begin{remark}[\citet{frieze1997}]
\label{remark: misdo2}
Let $\bar{Z} = \frac{k}{k - 1} Z - \frac{e e^T}{k - 1}$ with $Z$ be the decision variable in the $\misdo$ formulation~\eqref{eq:MI-SDO}. 
We can rewrite $\misdo$ formulation~\eqref{eq:MI-SDO} as follows.
    \begin{subequations}\label{eq:MI-SDO2}
	\begin{alignat}{3}
	\max\quad & \frac{(k-1)}{k}\sum_{\{u,v\} \in E} w_{uv}(1 - \bar{Z}_{uv}) \label{eq:misdo2-obj} \\
 	(\misdos) \quad \text{s.t.} \quad & \bar{Z}_{vv} = 1 & \forall v \in V \label{eq:misdo2-diag}\\
 	\quad &  \bar{Z} \succeq  0 \label{eq:misdo2-psd}\\
	 & \bar{Z} \in \bigg\{-\frac{1}{k - 1},1 \bigg\}^{n \times n}.
	\end{alignat}
\end{subequations}
\end{remark}

Interested readers are encouraged to refer to~\cite{ sotirov2014_efficient, van2016_new, Wang2020-k-partition, lu2021_branch} for more details on semidefinite optimization and mixed integer semidefinite optimization formulations of the max $k$-cut.
%

%
%

\section{A Theoretical Comparison of Relaxations} \label{section: theoretical comparison}
In this section, we provide theoretical comparisons between the continuous relaxations of $\bqo$ formulation~\eqref{eq: maxKcut_qp} and formulations~\eqref{eq:maxKcut}-\eqref{eq:MI-SDO}.
For analysis purposes, we introduce $y$ variables to the $\bqo$ formulation~\eqref{eq: maxKcut_qp}: 
for every edge $\{u,v\} \in E$, variable $y_{uv}$ equals one if $\{u,v\}$ is a cut edge.
\begin{equation}
    y_{uv} = 1 - \sum_{j \in P} x_{uj}x_{vj}  \qquad \forall \{u,v\} \in E \label{eq: maxKcut_qp y variable definition}. 
\end{equation}
Furthermore, we define the set of lifted continuous relaxation of the $\bqo$ formulation as follows.
\begin{align*}
    \relaxbqov \coloneqq \bigg\{&(x,y) \in [0,1]^{n \times k} \times \mathbb{R}^m \ \Big| \ \\
    &(x,y)\text{ satisfies constraints~\eqref{eq: maxKcut_qp_2}~\text{and }\eqref{eq: maxKcut_qp y variable definition}}\bigg\}.
\end{align*}

The following remark shows that we do not need to impose 0-1 bounds on $y$ variables. 

\begin{remark}
	Constraints $y \in [0,1]^m$ are implied by the $\bqo$ formulation~\eqref{eq: maxKcut_qp} and constraints~\eqref{eq: maxKcut_qp y variable definition}.
 \end{remark}

 To see this, consider  a point $(\hat{x}, \hat{y}) \in \relaxbqov$.  
For every edge $\{u,v\} \in E$, we have
\begin{align*}
    \hat{y}_{uv} = 1 - \sum_{j \in P} \hat{x}_{uj} \hat{x}_{vj} \ge 1 - \sum_{j \in P} \hat{x}_{uj} = 1 - 1 = 0.
\end{align*}
The first equality holds by constraints~\eqref{eq: maxKcut_qp y variable definition}. 
The inequality holds because, for any partition $j \in P$, we have $x_{vj} \le 1$. 
The second equality holds by constraints~\eqref{eq: maxKcut_qp_2}.
Furthermore, we have
\begin{align*}
    \hat{y}_{uv} = 1 - \sum_{j \in P} \hat{x}_{uj}\hat{x}_{vj} \le 1 - 0 = 1.
\end{align*}
The first equality holds by constraints~\eqref{eq: maxKcut_qp y variable definition}. 
The inequality holds because for any partition $j \in P$, and any vertex $v \in V$, we have $x_{vj} \ge 0$. 
%
%

Next, we compare the strength of the continuous relaxation of the BQO formulation~\eqref{eq: maxKcut_qp} against the continuous relaxations of the other formulations presented above.

\subsection{$\bqo$ vs. $\vmilo$}
First, we prove Lemma~\ref{lemma: bilinear valid inequality} that will be used in our further analyses. 
We define $[n] \coloneqq \{1, \dots, n\}$ for every $n \in \mathbb{Z}_{++}$.
\begin{lemma}\label{lemma: bilinear valid inequality}
Let $a \in [0, 1]^n$ with $n \ge 2$.
Then, we have
\begin{equation}\label{eq:bilinear valid inequality}
        1 - \sum_{i\in [n]} a_i + \sum_{\{i,j\}\in \binom{[n]}{2}} a_i a_j\geq 0.
\end{equation}
\end{lemma}

\begin{proof}

We prove the claim by induction.
First, we show that the inequality holds for the base case $n=2$. 
In this case, we have
\begin{align}\label{eq:bilinear valid inequality simple case}
 1 - a_1 - a_2 + a_1 a_2 = (1- a_1) (1- a_2) \ge 0.
\end{align}
The inequality~\eqref{eq:bilinear valid inequality simple case} holds because for every $i \in \{1, 2\}$, we have $1- a_i \geq 0$.

Now suppose that inequality~\eqref{eq:bilinear valid inequality} holds for $n=s \ge 2$ (induction hypothesis).
It suffices to show that it also holds for $n=s+1$.
\begin{subequations}
\begin{alignat}{2}
0 & \leq 
\Biggpl 1 -  \sum_{i\in [s]} a_i + \sum_{\{i,j\}\in \binom{[s]}{2}} a_i a_j 
\Biggpr
(1 - a_{s+1}) \label{eq:bilinear valid inequality-a}\\
& = 1 - \sum_{i\in [s+1]} a_i + \sum_{\{i,j\}\in \binom{[s+1]}{2}} a_i a_j - a_{s+1}\sum_{\{i,j\}\in \binom{[s]}{2}} a_i a_j  \label{eq:bilinear valid inequality-b}\\
& \leq 1 - \sum_{i\in [s+1]} a_i + \sum_{\{i,j\}\in \binom{[s+1]}{2}} a_i a_j. \label{eq:bilinear valid inequality-c}
\end{alignat}
\end{subequations}
Inequality~\eqref{eq:bilinear valid inequality-a} holds by induction hypothesis and because $1 - a_{s+1} \geq 0$.
Inequality~\eqref{eq:bilinear valid inequality-c} holds as $-a_{s+1}\sum\limits_{\{i,j\}\in \binom{[s]}{2}} a_i a_j \leq 0$.
%
\end{proof}

Now we define the polytope of the continuous relaxation of the $\vmilo$ formulation~\eqref{eq:maxKcut} as follows.
\begin{align*}
 \relaxvmilo \coloneqq \bigg\{&(x,y) \in [0,1]^{n \times k } \times [0,1]^{ m} \ \Big| \ \\
 &(x,y) \text{ satisfies constraints~\eqref{eq:clusterCons}--\eqref{eq:necessaryCons}}\bigg\}.
\end{align*}

Theorem~\ref{proposition: bqo relaxation vs milo} shows that the continuous relaxation of the $\bqo$ formulation is stronger than that of the $\vmilo$ formulation.
\begin{theorem} \label{proposition: bqo relaxation vs milo}
 $ \relaxbqov \subset \relaxvmilo$.
\end{theorem}

\begin{proof}
	Let point $(\hat{x}, \hat{y}) \in \relaxbqov$. 
	First, we are to show that $(\hat{x}, \hat{y}) \in \relaxvmilo$.
	We show that $(\hat{x}, \hat{y})$ satisfies constraints~\eqref{eq:reduncantCon1}. For every edge $\{u,v\} \in E$ and every partition $j \in P$, we have

	\begin{subequations}
		\begin{alignat}{2}
	\hat{y}_{uv} &= 1 - \sum_{i \in P} \hat{x}_{ui} \hat{x}_{vi} \label{eq: maxKcut relaxation a}\\
	&= \sum_{i \in P} \hat{x}_{ui} - \sum_{i \in P} \hat{x}_{ui} \hat{x}_{vi} \label{eq: maxKcut relaxation b}\\
	&= \hat{x}_{uj} + \sum_{i \in P \setminus \{j\}} \hat{x}_{ui} - \hat{x}_{uj} \hat{x}_{vj} - \sum_{i \in P\setminus \{j\}} \hat{x}_{ui} \hat{x}_{vi} \label{eq: maxKcut relaxation c}\\
	&\ge \hat{x}_{uj} + \sum_{i \in P \setminus \{j\}} \hat{x}_{ui} - \hat{x}_{vj} - \sum_{i \in P \setminus \{j\}} \hat{x}_{ui} \hat{x}_{vi} \label{eq: maxKcut relaxation d}\\
	&= \hat{x}_{uj} - \hat{x}_{vj} + \sum_{i \in P \setminus \{j\}} \hat{x}_{ui} (1 - \hat{x}_{vi}) \label{eq: maxKcut relaxation e}\\
	&\ge \hat{x}_{uj} - \hat{x}_{vj}. \label{eq: maxKcut relaxation f}
	\end{alignat}
	\end{subequations}
	
	Equality \eqref{eq: maxKcut relaxation a} holds by constraints~\eqref{eq: maxKcut_qp y variable definition}. 
	Equality \eqref{eq: maxKcut relaxation b} follows from constraint \eqref{eq: maxKcut_qp_2}.
	Inequality \eqref{eq: maxKcut relaxation d} holds because $\hat{x}_{uj} \leq 1$. 
	Inequality \eqref{eq: maxKcut relaxation f} holds because $\sum_{i \in P \setminus \{j\}} \hat{x}_{ui} (1 - \hat{x}_{vi}) \geq 0$. 
	%
	
	Finally, we show that $(\hat{x}, \hat{y})$ satisfies constraints~\eqref{eq:necessaryCons}. 
	For every edge $\{u,v\} \in E$ and every partition $j \in P$, we have
	\begin{subequations}
		\begin{alignat}{2}
	\hat{y}_{uv} &= 1 - \sum_{i \in P} \hat{x}_{ui} \hat{x}_{vi}  \label{eq: maxKcut relaxation necessaryCons a}\\
	&=1 - \hat{x}_{uj}\hat{x}_{vj} - \sum_{i \in P \setminus \{j\}} \hat{x}_{ui} \hat{x}_{vi} \label{eq: maxKcut relaxation necessaryCons b}\\
	&\le 2 - 1 - \hat{x}_{uj}\hat{x}_{vj}  \label{eq: maxKcut relaxation necessaryCons c}\\
	&\leq 2 - (\hat{x}_{uj} + \hat{x}_{vj}).  \label{eq: maxKcut relaxation necessaryCons f}
	\end{alignat}
	\end{subequations}
	Equality \eqref{eq: maxKcut relaxation necessaryCons a} holds by constraints~\eqref{eq: maxKcut_qp y variable definition}. 
	Inequality \eqref{eq: maxKcut relaxation necessaryCons c} holds by inequality~\eqref{eq:bilinear valid inequality simple case} in Lemma~\ref{lemma: bilinear valid inequality}  and $\sum\limits_{i \in P \setminus \{j\}} \hat{x}_{ui} \hat{x}_{vi} \geq~0$.

	Now, we are to show that there exists a point $(\hat{x}, \hat{y}) \in \relaxvmilo$ such that $(\hat{x}, \hat{y}) \not \in\relaxbqov$.
	For every $v \in V$, let $\hat{x}_{v1}=\hat{x}_{v2} =0.5$. For every vertex $v \in V$ and every partition $j \in \{3,4,\dots, k\}$, we define $\hat{x}_{vj}=0$. 
	Furthermore, for every edge $\{u,v\} \in E$, we define $\hat{y}_{uv} = 1$. 
        It is simple to check that $(\hat{x}, \hat{y})$ in $\relaxvmilo$.
	So, point $(\hat{x}, \hat{y}) \in \relaxvmilo \setminus \relaxbqov$ because $(\hat{x}, \hat{y})$ violates constraints~\eqref{eq: maxKcut_qp y variable definition}. 
	Thus, the proof is complete.
\end{proof}

The following remark shows the $\vmilo$ formulation has a weak relaxation.

\begin{remark}\label{remark: optimal_obj_vmilo}
The optimal objective of the continuous relaxation for the $\vmilo$ formulation~\eqref{eq:maxKcut} is equal to $\sum_{\{u,v\}\in E} \max\{w_{uv}, 0\}$. 
\end{remark}

To see this,  note that an optimal solution $(x^*,y^*)$ for the continuous relaxation of the $\vmilo$ formulation~\eqref{eq:maxKcut} is obtained by setting $x^*_{vj} = \frac{1}{k}$ for every $v \in V$ and $j\in P$. 
Also for every edge $\{u,v\}\in E$, we set $y^*_{uv}$ to $1$ if $w_{uv} > 0$ and $0$ otherwise.
\subsection{$\bqo$ vs. $\emilo$}
To conduct a theoretical comparison between the continuous relaxations of the $\bqo$ and $\emilo$ formulations, we lift the dimensionality of the $\bqo$ by introducing new $z$ variables.
\begin{equation}\label{eq:z variable definition}
    z_{uv} \coloneqq \sum_{j \in P} x_{uj}x_{vj}, \qquad \forall \{u,v\} \in {\binom{V}{2}}.
\end{equation}
We define the set of lifted continuous relaxation of the $\bqo$ formulation in $z$-space as follows.
\begin{align*}
    \relaxbqoe \coloneqq \bigg\{&(x,z) \in [0,1]^{n \times k} \times \mathbb{R}^{{\binom{n}{2}}} \ \Big| \ \\
    & (x,z)\text{ satisfies constraints~\eqref{eq: maxKcut_qp_2}~\text{and }\eqref{eq:z variable definition}}\bigg\}.
\end{align*}

We also define the polytope of the $\emilo$ formulation as follows.
\begin{align*}
 \relaxemilo \coloneqq \bigg\{z \in [0,1]^{\binom{n}{2}} \ \Big| \ z \text{ satisfies constraints~\eqref{eq:triangle}--\eqref{eq:atmostk}}\bigg\}.
\end{align*}

We show that the continuous relaxation of a projection of the lifted $\bqo$ formulation on the $z$ space is stronger than that of the $\emilo$ formulation.
\begin{theorem} \label{proposition: bqo relaxation vs E-MILO}
 $\proj_{z}\relaxbqoe \subset \relaxemilo$ for $n > k$.
\end{theorem}

\begin{proof}
Consider a point $(\hat{x}, \hat{z}) \in \relaxbqoe$. 
We are to show that $\hat{z} \in \relaxemilo$. 
For every set $\{u,v,w\} \subseteq V$, we show that point $\hat{z}$ satisfies constraints~\eqref{eq:triangle}. 
\begin{subequations}
\begin{align}
    \hat{z}_{uv} + \hat{z}_{vw} &= \sum_{j \in P} \hat{x}_{uj} \hat{x}_{vj} + \sum_{j \in P} \hat{x}_{vj} \hat{x}_{wj} \label{eq:comp1}\\
    &=\sum_{j \in P} \hat{x}_{vj} (\hat{x}_{uj} + \hat{x}_{wj}) \label{eq:comp2}\\
    &\le \sum_{j \in P} \hat{x}_{vj}(1+\hat{x}_{uj} \hat{x}_{wj}) \label{eq:comp3}\\
    & = \sum_{j \in P} \hat{x}_{vj} + \sum_{j \in P} \hat{x}_{vj}(\hat{x}_{uj} \hat{x}_{wj}) \label{eq:comp4}\\
    & = 1 + \sum_{j \in P} \hat{x}_{vj}(\hat{x}_{uj} \hat{x}_{wj})\label{eq:comp5}\\
    & \le 1 + \sum_{j \in P} \hat{x}_{uj} \hat{x}_{wj}\label{eq:comp6}\\
    & = 1 + \hat{z}_{uw}. \label{eq:comp7}
\end{align}
\end{subequations}
Equality~\eqref{eq:comp1} holds by definition~\eqref{eq:z variable definition}.
Inequality~\eqref{eq:comp3} holds by inequality~\eqref{eq:bilinear valid inequality simple case} in Lemma~\ref{lemma: bilinear valid inequality}.
Equality~\eqref{eq:comp5} holds by constraints~\eqref{eq: maxKcut_qp_2}.
Inequality~\eqref{eq:comp6} holds by the fact that $\hat{x}_{vj} \le 1$.
Equality~\eqref{eq:comp7} holds by definition~\eqref{eq:z variable definition}.

Furthermore, we show that point $\hat{z}$ satisfies constraints~\eqref{eq:atmostk}. 
For every vertex set $Q \subseteq V$ with $|Q| = k + 1$, we have
\begin{subequations}
\begin{align}
    \sum_{\{u,v\} \in \binom{Q}{2}} \hat{z}_{uv} &= \sum_{\{u,v\} \in \binom{Q}{2}} \sum_{j \in P} \hat{x}_{uj} \hat{x}_{vj} \label{eq:q1}\\
    &=\sum_{j \in P} \Biggpl \sum_{\{u,v\} \in \binom{Q}{2}} \hat{x}_{uj} \hat{x}_{vj} \Biggpr \label{eq:q2}\\
    &\ge \sum_{j \in P} \Biggpl \sum_{u \in Q} \hat{x}_{uj} - 1\Biggpr \label{eq:q3}\\
    &= \sum_{u \in Q} \sum_{j \in P} \hat{x}_{uj} - k \label{eq:q4}\\
    & = k+1-k = 1\label{eq:q5}.
\end{align}
\end{subequations}
Equality~\eqref{eq:q1} holds by definition~\eqref{eq:z variable definition}. 
Inequality~\eqref{eq:q3} holds by Lemma~\ref{lemma: bilinear valid inequality}. 
Equality~\eqref{eq:q5} holds by constraints~\eqref{eq: maxKcut_qp_2} and because $|Q| = k+1$.

Finally, for every $\{u,v\} \in \binom{V}{2}$, we show that $0 \le \hat{z}_{uv} \le 1$. 
Because for every vertex $v \in V$ and every partition $j \in P$ we have $\hat{x}_{vj} \ge 0$, it follows that $\hat{z}_{uv} \geq 0$. 
For every $\{u,v\} \in \binom{V}{2}$, we show that $\hat{z}_{uv} \le 1$.
\begin{align*}
    \hat{z}_{uv} = \sum_{j \in P} \hat{x}_{uj} \hat{x}_{vj} \le \sum_{j \in P} \hat{x}_{uj} = 1. 
\end{align*}
The first equality holds by definition~\eqref{eq:z variable definition}. 
The inequality holds because $\hat{x}_{vj} \le 1$ for every vertex $v \in V$ and every partition $j \in P$. The last equality holds by constraints~\eqref{eq: maxKcut_qp_2}. This implies that $\proj_{z}\relaxbqoe \subseteq \relaxemilo$.

Now we show that the inclusion is strict for any non-trivial instance of the max $k$-cut problem with $n>k$.   
For every $\{u,v\} \in \binom{V}{2}$, we define $\hat{z}_{uv}$ as a point that belongs to the polytope of the $\emilo$ formulation; that is, $\hat{z} \in \relaxemilo$.
\begin{equation*}
    \hat{z}_{uv} \coloneqq  \frac{2}{k(k+1)}.
\end{equation*}
For every vertex $v \in V$, let $\mathbf{x}_v\in [0,1]^k$ be the assignment vector of vertex $v$. 
By definition~\eqref{eq:z variable definition}, we have
\begin{equation}\label{eq:z based on dot product of x}
    \hat{z}_{uv} =  \hat{\mathbf{x}}_u^T \hat{\mathbf{x}}_v = \|\hat{\mathbf{x}}_u \|_2 \|\hat{\mathbf{x}}_v \|_2 \cos{\hat{\theta}_{uv}}.
\end{equation}
By constraints~\eqref{eq: maxKcut_qp_2}, we have $\|\hat{\mathbf{x}}_v \|_1 = 1$. 
Then for every vertex $v \in V$, we have
\begin{equation} \label{eq:norm 2 of x bounds}
    \frac{1}{\sqrt{k}} \leq \| \hat{\mathbf{x}}_v \|_2  \leq 1.
\end{equation}
The first inequality holds because $\| \hat{\mathbf{x}}_v \|_2$ reaches its minimum when $\hat{x}_{vj} = \frac{1}{k}$ for every partition $j \in P$. 
%

By lines~\eqref{eq:z based on dot product of x} and~\eqref{eq:norm 2 of x bounds}, we have 
\begin{equation*}
    \frac{2}{k(k+1)} \leq \cos{\hat{\theta}_{uv} } \leq \frac{2}{k+ 1}.
\end{equation*}
For every $\{u,v\} \in \binom{V}{2}$, this implies that we have the following relations because $k \ge 2$.
\begin{equation}\label{eq:bounds on angle}
    {\arccos{\Biggfl\frac{1}{\sqrt{k}}\Biggfr}} < \arccos{\Biggfl \frac{2}{k+ 1}\Biggfr} \leq \hat{\theta}_{uv} \leq \arccos{\Biggfl\frac{2}{k(k+1)}\Biggfr }.
\end{equation}
Consider $k+1$ vectors in $\mathbb{R}^k_+$ and let $\theta_{\min}$ be the minimum angle between all vector pairs.
It follows that the maximum value of $\theta_{\min}$ is $\arccos{\Bigfl \frac{1}{\sqrt{k}}\Bigfr}$. 
This case happens when $k$ vectors are located on the axes, and one vector is located at the center of the positive orthant.
Without loss of generality, consider $k$ unit vectors on $k$ different axes in $\mathbb{R}^k_+$  
and a vector with all entries equal to $\frac{1}{\sqrt{k}}$.
For example, the maximum values of $\theta_{\min}$ are $45^{\circ}$ and $\arccos{\Bigfl \frac{1}{\sqrt{3}}\Bigfr}\approx 54.7^{\circ}$ for $k=2$ and $k=3$, respectively.  
%

Since all vectors $\hat{\mathbf{x}}_v$ are in the positive orthant and $n>k$, there are vectors $\hat{\mathbf{x}}_a$ and $\hat{\mathbf{x}}_b$ with $ \hat{\theta}_{ab} \leq \arccos{\Bigfl \frac{1}{\sqrt{k}}\Bigfr}$. 
However, this contradicts the first inequality of line~\eqref{eq:bounds on angle}.
Thus, there is no feasible solution of the $\bqo$ formulation that satisfies definition~\eqref{eq:z variable definition}.
This completes the proof. 
\end{proof}

\citet{Wang2020-k-partition} prove that their reduced $\emilo$ ($\remilo$) formulation is as strong as the projection of the $\emilo$ formulation on the edges of an extended chordal graph. 
Further, their computational experiments show the superiority of the $\remilo$ over the $\emilo$ for sparse chordal graphs.
\subsection{$\bqo$ vs. {\normalfont MISDO}}
Now we define the continuous relaxation of the $\misdo$ formulation~\eqref{eq:MI-SDO} as follows.
\begin{align*}
 \relaxmisdo \coloneqq \Big\{Z \in [0,1]^{n\times n} \ \big| \ Z \text{ satisfies constraints~\eqref{eq:misdo-diag}--\eqref{eq:misdo-psd}}\Big\}.
\end{align*}
To conduct a theoretical comparison between the continuous relaxations of the BQO and MISDO formulations, we lift the dimensionality of the BQO formulation by introducing a new symmetric matrix $Z \in \mathbb{R}^{n \times n}$ defined as follows.
\begin{equation}\label{eq:Zdef}
    Z\coloneqq D^x + \sum_{j\in P} \mathbf{x}_j \mathbf{x}_j^T,
\end{equation}
with $D^x$ be a diagonal matrix and $D^x_{vv} = 1 - \sum_{j \in P}x_{vj}^2$.
For every partition $j\in P$, we redefine vector $\mathbf{x}_j\in [0,1]^n$ such that $\mathbf{x}_{jv} = x_{vj}$.
For comparison purposes, we also define $\relaxbqom$.
\begin{align*}
 \relaxbqom \coloneqq \bigg\{&(x,Z) \in [0,1]^{n \times k} \times \mathbb{R}^{n \times n} \ \Big| \ \\
 & (x,Z)\text{ satisfies constraints~\eqref{eq: maxKcut_qp_2}~\text{and}~\eqref{eq:Zdef}}\bigg\}.
\end{align*}

\citet{Eisenblater2002_Semidefinite} conducts a set of experiments on a semidefinite formulation, which was developed by~\citet{frieze1997} (see Remark~\ref{remark: misdo2}), and show the tightness of the continuous relaxation of their formulation computationally.
However, they declare that continuous relaxations of the semidefinite optimization and $\emilo$ formulations have feasible sets whose union is larger than that of any of the feasible sets (i.e., it is not straightforward to compare them). 
Furthermore,~\citet{de2019} propose $\misdo$-based constraints for the $\emilo$ formulation to strengthen its relaxation.
The following theorem compares the relaxation strength of the $\misdo$ formulation against the $\bqo$ formulation.
\begin{theorem} \label{proposition: bqo relaxation vs misdo}
 $\proj_{Z}\relaxbqom \subseteq \relaxmisdo$.
\end{theorem}

\begin{proof}
For any fractional solution $x\in [0,1]^{n \times k}$, we have  $\sum_{j \in P}x_{vj}^2 \le 1$ by constraints~\eqref{eq: maxKcut_qp_2}.
Thus, $k D^x \succeq 0$ and $Z \in [0,1]^{n \times n}$ by definition~\eqref{eq:Zdef}.
Matrix $k \sum_{j\in P} \mathbf{x}_j \mathbf{x}_j^T -  e e^T$ is positive-semidefinite if and only if for every vector $b\in \mathbb{R}^n$, we have 
\begin{equation*}
    b^T \Biggfl k \sum_{j\in P} \mathbf{x}_j \mathbf{x}_j^T -  e e^T \Biggfr b \ge  0.
\end{equation*}
We define 
\begin{align}\label{eq:defOne}
    \beta_j \coloneqq b^T \mathbf{x}_j~~\forall j \in P,~\text{and}~\alpha \coloneqq b^T e. 
\end{align}
It suffices to show $k \sum_{j\in P} \mathbf{x}_j \mathbf{x}_j^T - e e^T \succeq 0$.
We can rewrite constraints~\eqref{eq: maxKcut_qp_2} as $\sum_{j\in P} \mathbf{x}_j = e$.
Thus, we have 
\begin{equation}\label{eq:defTwo}
    \sum_{j\in P} \beta_j = \sum_{j\in P}  b^T \mathbf{x}_j = b^T \sum_{j\in P} \mathbf{x}_j = b^T e = \alpha.
\end{equation}
So, we have that
\begin{align*}
    b^T \Biggfl k\sum_{j\in P} \mathbf{x}_j \mathbf{x}_j^T -  e e^T \Biggfr b 
    &= k \sum_{j \in P} \beta_j^2 - \alpha^2 \\
    &= k\sum_{j \in P} \beta_j^2 - \Biggfl \sum_{j \in P} \beta_j {\Biggfr}^2 \\
    %
    %
    %
    %
    & = \sum_{ \{i,j\}\in \binom{P}{2} } \big(\beta_i - \beta_j {\big)}^2 \ge 0.
\end{align*}
The first equality holds by definitions~\eqref{eq:defOne}.
The second equality holds by line~\eqref{eq:defTwo}.
This completes the proof.
\end{proof}

By Theorem~\ref{proposition: bqo relaxation vs misdo} and Remark~\ref{remark: misdo2}, we have the following corollary.

\begin{corollary}
$\proj_{\bar{Z}}\relaxbqomii \subseteq \relaxmisdoii$ with
\begin{align*}
 \relaxbqomii \coloneqq \bigg\{&(x,\bar{Z}) \in [0,1]^{n \times k} \times \mathbb{R}^{n \times n} \ \Big| \ (x,\bar{Z})\text{ satisfies constraints~\eqref{eq: maxKcut_qp_2}}\\
 & \text{and } \bar{Z} \coloneqq \frac{1}{k-1}\Big(kD^x - e e^T + k\sum_{j\in P} \mathbf{x}_j \mathbf{x}_j^T\Big)\bigg\}.
\end{align*}
\end{corollary}

\section{Computational Experiments} \label{section: computational experiments}
In this section, we conduct a set of experiments to evaluate our theoretical results in practice.
In better words, we compare the relaxations of the discussed formulations using state-of-the-art solvers.
We run the computational experiments on a machine with Dual Intel Xeon{\textregistered} CPU E5-2630~@ 2.20 GHz (20 cores) and 64 GB of RAM. 
We have developed the Python package \texttt{MaxKcut}~\citep{Fakhimi2022_maxkcut} to conduct the computational experiments. 
We employ Gurobi 10.0.0~\citep{gurobi} to solve our mixed integer optimization formulations.
We also use MOSEK~10.0.27~\citep{mosek} to run our experiments for the $\misdos$ formulation.
In both Gurobi and MOSEK solvers, we set the number of threads and time limit  to 10 and 3,600 seconds, respectively.
The code, instances, and results are available on GitHub~\citep{Fakhimi2022_maxkcut}.

We run our experiments on the following sets of instances: (i) band~\citep{Wang2020-k-partition}, (ii) spinglass~\citep{Wang2020-k-partition}, (iii) Color02~\citep{color2}, and (iv) Steiner-160~\citep{koch2001} (in total 97 different instances).
We classify these instances into different classes based on the number of vertices (from 50 to 250) and their density (from 5 percent to 100 percent).
Thanks to the  spatial branch and bound algorithm, Gurobi can solve the continuous relaxation of the BQO formulation~\eqref{eq: maxKcut_qp} to global-$\epsilon$ optimality; however, we stop the solving process whenever the solver reaches the one-hour time limit. 
The same time limit is set for all solvers, as mentioned above.

Because of the large number of clique constraints in the $\remilo$ formulation (i.e., a sparse variant of the $\emilo$ model~\eqref{eq:E-MILO} proposed by~\citet{Wang2020-k-partition}) for large-scale instances, it is not practical to add all of them upfront. 
Thus, we initially relax the clique constraints in the $\remilo$ formulation to avoid memory shortage. 
After solving the relaxed formulation to optimality, we iteratively add the most violated constraints using the dual Simplex method and solve the problem.
We stop whenever there is not enough free memory.
All the clique constraints are added upfront for small instances with $n \le 100$.
We note that the solving process of all instances with $n > 100$ reaches the time limit.

Figure~\ref{fig: geo-mean} summarizes our results with instances on the horizontal axis and the geometric mean of the scaled relaxed objective values on the vertical axis.
We scale the upper bounds by the best obtained upper bound to report the results.
By scaled relaxed objective, we mean that the best upper bound on a given instance is set to 1 and the other upper bounds are scaled by this best upper bound.
Furthermore, we use the geometric mean of the scaled upper bound for every batch of instances.
For $k \in \{3,4\}$, we solve the following formulations:
\begin{enumerate}[label=(\roman*)]
    \item the $\bqo$ formulation~\eqref{eq: maxKcut_qp};
    \item the continuous relaxation of the $\bqo$ formulation~\eqref{eq: maxKcut_qp};
    \item the continuous relaxation of the $\vmilo$ formulation~\eqref{eq:maxKcut};
    \item the continuous relaxation of the $\remilo$ formulation;
    \item the continuous relaxation of the MISDO formulation~\eqref{eq:MI-SDO2}. 
\end{enumerate}

Figure~\ref{fig: geo-mean} shows the superiority of the $\bqo$ formulation and its  relaxed variant over both $\vmilo$ and $\remilo$ formulations when $k \in \{3,4\}$.
This observation matches the results of Theorems~\ref{proposition: bqo relaxation vs milo} and~\ref{proposition: bqo relaxation vs E-MILO}.
%
%
Although Theorem~\ref{theorem: BQO vs its relxation} implies that both the BQO formulation~\eqref{eq: maxKcut_qp} and its corresponding continuous relaxation solve the max $k$-cut problem, Figure~\ref{fig: geo-mean} shows that solving the $\bqo$ formulation provides a better upper bound in comparison with solving the continuous relaxation of the BQO formulation within the one-hour time limit.
%
We note that Gurobi's recent advancements in handling non-convex quadratic optimization problems enabled us to solve the BQO formulation.

Figure~\ref{fig: geo-mean} illustrates the inferiority of the relaxation of the $\vmilo$ formulation~\eqref{eq:maxKcut} over all other formulations for all sets of instances when $k \in \{3,4\}$.
This behavior is justifiable by Remark~\ref{remark: optimal_obj_vmilo}.
We also observe that the relaxed $\misdos$ formulation~\eqref{eq:MI-SDO2} almost matches the best upper bound provided by solving the BQO formulation~\eqref{eq: maxKcut_qp}.
For sparse instances, it performs similarly to the relaxed $\remilo$ formulation.
In our experiments, we chose to solve the continuous relaxation of the $\misdos$ formulation~\eqref{eq:MI-SDO2} instead of the $\misdo$ formulation~\eqref{eq:MI-SDO} since MOSEK can take advantage of the sparsity of the $\misdos$ formulation.
MOSEK employs the {\em interior point method} (IPM) to solve the semidefinite optimization formulations~\cite{mosek}.
For a given instance with $n$ vertices, the IPM requires the solution of a linear system in $\mathbb{R}^{\mathcal{O}(n^2) \times \mathcal{O}(n^2) }$ at every iteration.
These extremely large linear systems exhaust all the memory as the algorithm converges to an optimal solution, and their condition numbers grow~\cite{Roos2005_Interior}. 
Regardless of the decent performance of the $\misdo$ formulation~\eqref{eq:MI-SDO} on small and medium-sized instances, it is not practical for solving instances with more than 200 vertices.
However, MOSEK solves larger instances of continuous relaxation of the $\misdos$ formulation~\eqref{eq:MI-SDO2}.
\begin{figure*}[ht]
\centering
\begin{subfigure}[t]{1\textwidth}
	\centering
	\includegraphics[width =1 \linewidth, trim = {.1cm .1cm .1cm .1cm}, clip ]{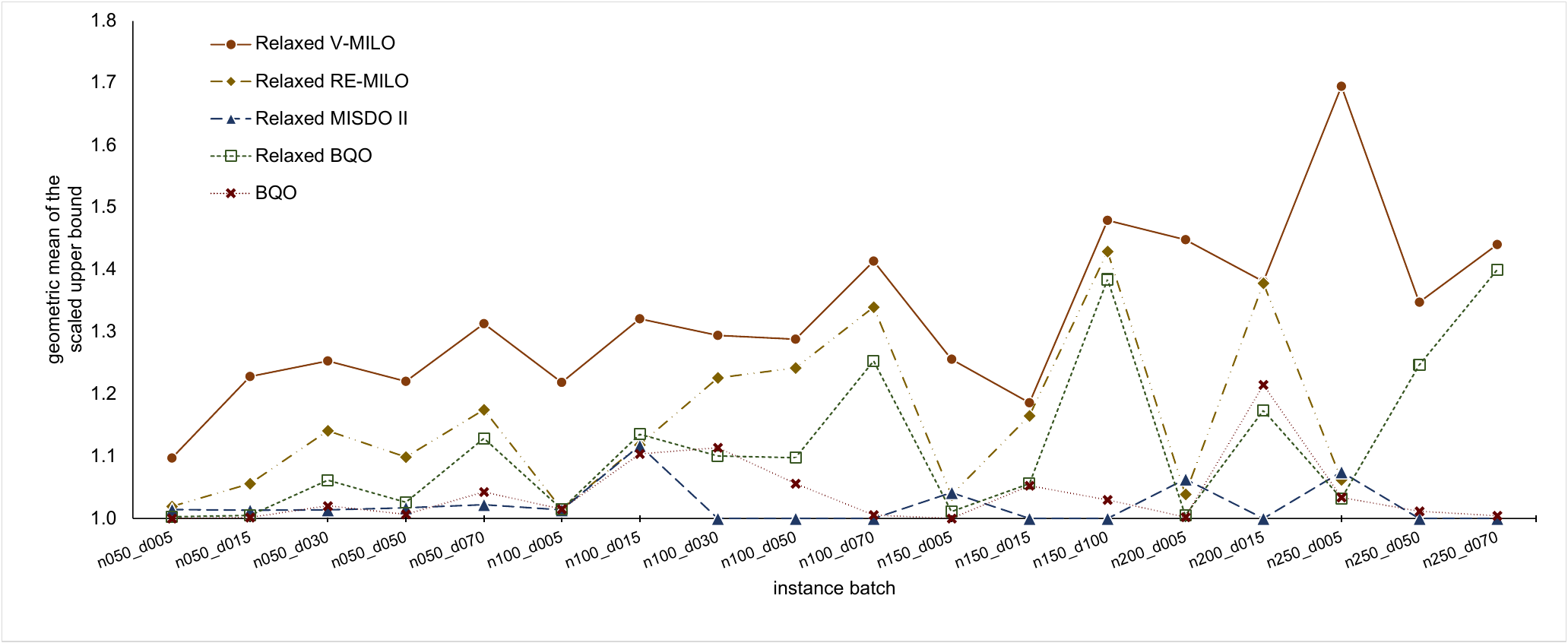}
	\caption{$k = 3$.}\label{fig: geo-mean k3}
\end{subfigure}\\
\begin{subfigure}[t]{1\textwidth}
	\centering
	\includegraphics[width =1 \linewidth, trim = {.1cm .1cm .1cm .1cm}, clip ]{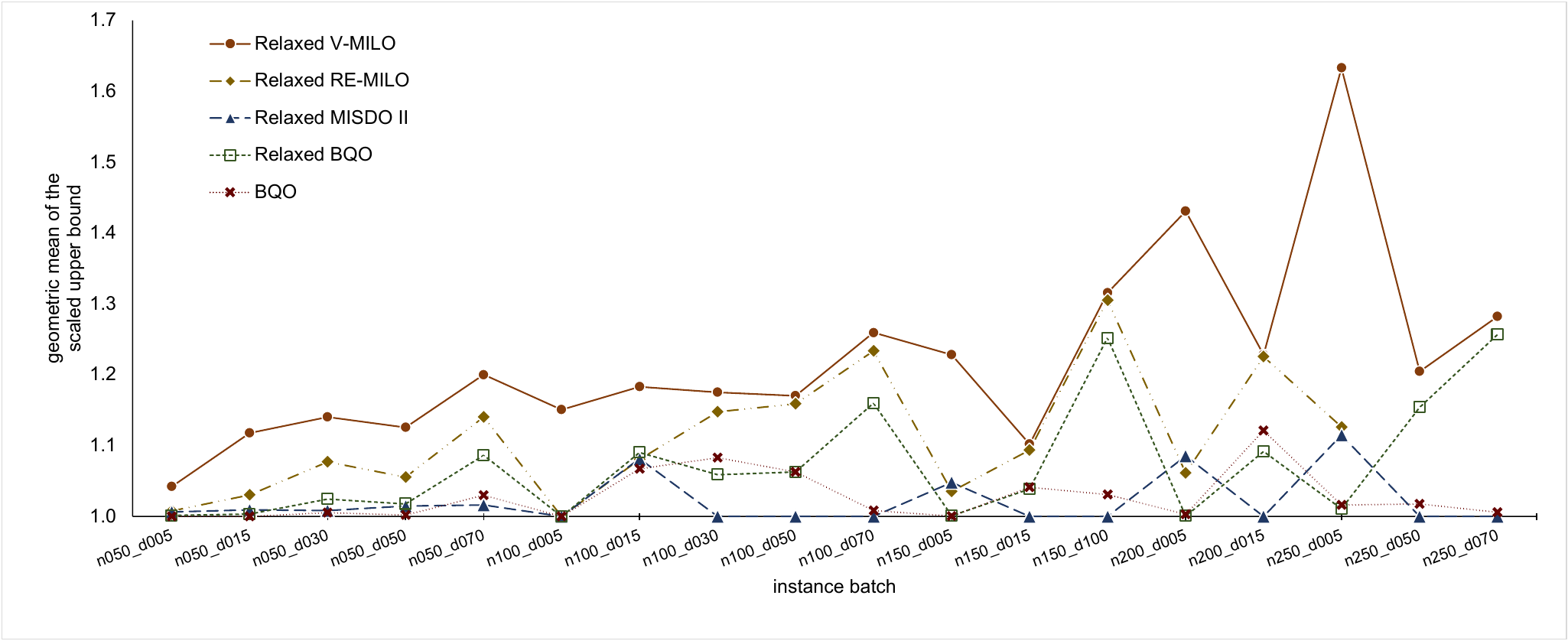}
	\caption{$k = 4$.}
	\label{fig: geo-mean k4}
\end{subfigure}
\caption{The geometric mean of the scaled upper bound obtained by different methods in an hour time limit ($n$ represents the number of vertices, and $d$ denotes the graph density in percentage).}\label{fig: geo-mean}
\end{figure*}

\section{Conclusion}
Motivated by the importance of continuous relaxation in the solution process of the mixed integer optimization formulations of NP-hard problems, we compared the continuous relaxations of four well-known formulations of a fundamental combinatorial optimization problem, that is, the max $k$-cut problem.   
We proved that the continuous relaxation of a binary quadratic optimization formulation is tighter than other existing formulations; specifically, vertex-based and edge-based mixed integer linear optimization formulations.
We observed that our numerical experiments support the theoretical results for the superiority of the continuous relaxation of the $\bqo$ formulation over the continuous relaxations of $\vmilo$ and $\emilo$ formulations.
As a direction of future work, one might be interested in comparing the nonconvex optimization formulations of other combinatorial optimization problems with convex counterparts.
While many believe that convex formulations may outperform non-convex formulations in practice, our results show that this might not always be the case, thanks to the improved capability of solvers in handling non-convex quadratic optimization formulations.

\section*{Acknowledgments}
This work is supported by the Defense Advanced Research Projects Agency (DARPA), ONISQ grant W911NF2010022 titled {\em The Quantum
Computing Revolution and Optimization: Challenges and Opportunities}.
This research also used resources from the Oak Ridge Leadership Computing Facility, which is a DOE Office of Science User Facility supported under Contract DE-AC05-00OR22725.
We would like to thank an anonymous reviewer for constructive comments that improved our manuscript.

\bibliographystyle{elsarticle-num-names} 
\bibliography{max-k-cut}



\end{document}